\documentclass[12pt]{amsart}
\usepackage[utf8]{inputenc}
\usepackage{amsmath}
\usepackage{amssymb}
\usepackage{mathrsfs}
\usepackage{pst-node}
\usepackage{mathtools}
\usepackage{amsthm}
\usepackage{tikz-cd}
\usetikzlibrary{cd}
\newtheorem{theorem}{Theorem}[section]
\newtheorem{lemma}[theorem]{Lemma}
\newtheorem{corollary}{Corollary}[theorem]
\newtheorem*{remark}{Remark}
\newtheorem{prop}{Proposition}[section]
\title{Generators for the  moduli space of parabolic bundle}
\author{Lisa Jeffrey and Yukai Zhang}
\address{Mathematics Department, University of Toronto, Toronto, Ontario, Canada}

\makeatletter
\def\bbordermatrix#1{\begingroup \m@th
  \@tempdima 4.75\p@
  \setbox\z@\vbox{%
    \def\cr{\crcr\noalign{\kern2\p@\global\let\cr\endline}}%
    \ialign{$##$\hfil\kern2\p@\kern\@tempdima&\thinspace\hfil$##$\hfil
      &&\quad\hfil$##$\hfil\crcr
      \omit\strut\hfil\crcr\noalign{\kern-\baselineskip}%
      #1\crcr\omit\strut\cr}}%
  \setbox\tw@\vbox{\unvcopy\z@\global\setbox\@ne\lastbox}%
  \setbox\tw@\hbox{\unhbox\@ne\unskip\global\setbox\@ne\lastbox}%
  \setbox\tw@\hbox{$\kern\wd\@ne\kern-\@tempdima\left[\kern-\wd\@ne
    \global\setbox\@ne\vbox{\box\@ne\kern2\p@}%
    \vcenter{\kern-\ht\@ne\unvbox\z@\kern-\baselineskip}\,\right]$}%
  \null\;\vbox{\kern\ht\@ne\box\tw@}\endgroup}
\makeatother

\begin{document}
\maketitle
\section{Introduction}
The purpose of this note is to find explicit representatives in deRham cohomology for the generators of the cohomology of the moduli space of parabolic bundles, analogous to the results of \cite{groupcoho} for the moduli space of vector bundles. Further we use the explicit generators to compute the intersection pairing of its cohomology.\\
We intend to extend this study to the situation where the group $G$ is noncompact (for example $G = PSL(2, {\bf R})$, which is important for Teichm\"uller space and the moduli space of curves).
The techniques will be different for noncompact $G$, but one  might at least be able to study the symplectic volume  and these explicit differential forms would be useful for that purpose.\\
In section 2, we briefly discuss the definition of the moduli space of parabolic bundles through the extended moduli space defined in \cite{groupcoho}. Further we also estimate the cohomology of the extended moduli space using equivariant cohomology.\\
In section 3, we quickly review the Chern-Weil homomorphism and Chern classes.\\
In section 4 and 5, we apply the Leray-Hirsch theorem to the principal bundle $M \to M/G$ and relate it with the  Chern classes of the bundle.\\
In section 6, we apply the results from section 4 and 5 to the moduli space of parabolic bundles to obtain the generators. We also write out the explicit differential forms corresponding to the generators.\\
In section 7, we use the explicit differential forms of the generators to calculate the intersection formula.

\section{The moduli space of parabolic bundles} \label{s2}
We introduce the definition of the moduli space of parabolic bundles and its relation with the moduli space of flat connections.\\
An easy way to construct the moduli spaces is using the extended moduli space defined in \cite{groupcoho}.
We define the extended moduli space to be: $$X_\beta = (\epsilon_R \times e_\beta)^{-1}(G) = \{(g,X) \in G^{2g} \times \mathfrak{g} | g_1g_2 g_1^{-1}g_2^{-1} \cdots g_{2g}^{-1} = \beta exp(X)\}.$$ We consider the $G$ adjoint action on $X_\beta$ with moment map $\mu = pr_2$ (project onto $\mathfrak{g}$).\\
We can define the moduli space of parabolic bundles as a symplectic quotient \cite{verlin}:
\begin{align*}
    \mathcal{M}_{\beta,1}(\Lambda) &= X_\beta //_\Lambda G\\
    &= \mu^{-1}(\mathcal O_\Lambda) /G \\
    & \text{(where $\mathcal O_\Lambda$ is the orbit of $\Lambda$ under the adjoint action)}
\end{align*}
Notice that $\mu^{-1}(\mathcal O_\Lambda)$ over $\mu^{-1}(\Lambda)$ is a fibre bundle with fibre $G/T$. Further $G$ acts transitively on the $\mathfrak{g}$ component and $T$ fixes the $\mathfrak{g}$ component for points in $\mu^{-1}(\Lambda)$, in other words, the $T$ action is closed in $\mu^{-1}(\Lambda)$. Thus we have the following identification:
\begin{align*}
    \mathcal{M}_{\beta,1}(\Lambda) &=\mu^{-1}(\mathcal O_\Lambda)/G\\
    & = \mu^{-1}(\Lambda)/T\\
    & \cong \mu^{-1}(0)/T\\
    &\text{(for $\Lambda$ close to $0$)}
\end{align*}
Notice that in the last step we used the normal form theorem (see \cite{symtech} section 40) for $\Lambda$ close to $0$. This identification is not canonical  -- in the case of the extended moduli space we need to fix a connection to construct such an  identification. \\
We can also define the moduli space of flat connections using the symplectic quotient:
$$\mathcal{M}_{\beta} = X_\beta //_0 G = \mu^{-1}(0)/G$$
Now we have the following fibre bundle with fibre $G/T$:
\begin{equation} \label{eq1}
  \mathcal{M}_{\beta,1}(\Lambda) = \mu^{-1}(0)/ T \to \mu^{-1}(0)/ G = \mathcal{M}_{\beta}  
\end{equation}
Before finding the explicit generators of $H^*(\mathcal{M}_{\beta,1}(\Lambda))$ we will use equivariant cohomology to get an upper bound for the number of generators.\\
First we rewrite $\mathcal{M}_{\beta,1}(\Lambda)$ using the shifting trick (see \cite{sympub}):
\begin{align*}
    \mathcal{M}_{\beta,1}(\Lambda) &=X_\beta //_\Lambda G\\
    &= (X_\beta \times \mathcal \mathcal O_\Lambda) //_0 G\\
    &\text{(where the new moment map is $\nu(x, Y) = \mu(x) - Y$)}
\end{align*}
Thus, we can apply the Kirwan map, since $X_\beta \times \mathcal O_\Lambda$ is compact:
$$H^*_G(X_\beta \times \mathcal O_\Lambda) \to H^*((X_\beta \times \mathcal O_\Lambda) //_0 G) \text{ is surjective}$$
By expanding the left hand side and using equivariant formality twice (see \cite{equifor}):
\begin{align}
    H^*_G(X_\beta \times \mathcal O_\Lambda) &= (S\mathfrak{g}^*)^G \otimes H^*(X_\beta \times \mathcal O_\Lambda)\nonumber\\
    &= (S\mathfrak{g}^*)^G \otimes H^*(X_\beta) \otimes H^*(\mathcal O_\Lambda)\nonumber\\
    &= H^*_G(X_\beta) \otimes H^*(\mathcal O_\Lambda)
\end{align}
$H^*_G(X_\beta) \otimes H^*(\mathcal O_\Lambda)$ is an upper bound of $H^*(\mathcal{M}_{\beta,1}(\Lambda))$. The generators of $H^*_G(X_\beta)$ are provided in \cite{groupcoho} related to the generator of $H^*(\mathcal{M}_\beta)$ and the cohomology of $\mathcal O_\Lambda = G/T$ is well known. So it is intuitive to see that the generators of $H^*(\mathcal{M}_{\beta,1}(\Lambda))$ come from the generators of $H^*(\mathcal{M}_\beta)$ and generators of $H^*(G/T)$. In the following chapters we will develop a rigorous way to prove this and provide the explicit generators.\\

\section{Notation of Chern-Weil homomorphism}
Let $\pi : P \to M$ be a principal $G$-bundle, and let $\theta \in \Omega^1(P, \mathfrak{g})$ be the connection 1-form. Let $f : \mathfrak{g} \to \mathbb{R}$ be an  $Ad(G)$-invariant polynomial of degree $k$.\\
WLOG, let $\bar{f} : \mathfrak{g}^k \to \mathbb{R}$ be a multi-linear function and $f(X) = \bar{f}(X,X,\cdots, X)$.\\
For $\Omega = d\theta + \frac{1}{2}[\theta, \theta]$ the curvature, we can define:
$$f(\Omega) = \bar{f}(\Omega, \cdots, \Omega) \in \Omega^{2k}(P)$$
Explicitly we define:
$$f(\Omega)(v_1, \cdots, v_{2k}) = \sum_{\sigma \in S_{2k}} sgn(\sigma) \bar{f}(\Omega(v_{\sigma(1)},v_{ \sigma(1)}), \cdots , \Omega(v_{\sigma(2k-1)}, v_{\sigma(2k)}))$$

\begin{theorem}[Chern-Weil homomorphism] 
\label{CWmorphism}
For an invariant polynomial of degree $k$ and a  2-form $\Omega$ which is the 
curvature associated to  the connection $\theta$, we have: 
\[ \exists \Lambda \in \Omega^{2k}(M), \text{ such that } \pi^*(\Lambda) = f(\Omega)\]
Furthermore, the class $[\Lambda] \in H^{2k}(M)$ is independent of the connection $\theta$.
\end{theorem}
\newcommand{\Inv}{{\rm Inv}}
The following map, where we define $\Inv(\mathfrak{g}) = S(\mathfrak{g}^*)^G$ to be the space of invariant polynomials, is called the Chern-Weil homomorphism.
\begin{align*}
    w : \Inv(\mathfrak{g}) &\to H^*(M)\\
    f & \mapsto [\Lambda]
\end{align*}
\begin{remark}
Since $H^*(BG) = H^*_G(pt) = \Inv(\mathfrak{g}) \subset S(\mathfrak{g}^*)$, the Chern-Weil homomorphism is actually a homomorphism $H^*(BG) \to H^*(M)$.
\end{remark}

\section{Leray–Hirsch theorem and fibre bundle $M/T \to M/G$}
Now, let's review the Leray–Hirsch theorem (see Hatcher \cite{AT} p.432):
\begin{theorem}[Leray–Hirsch theorem]
\label{LH}
Suppose we have a fibre bundle $$F \xhookrightarrow{i} E \xrightarrow{p} B.$$
Suppose $H^n(F)$ is finitely generated and there exist $c_1, \cdots ,c_k \in H^*(E)$ so that for any fibre inclusion $i$ we have:
$$i^*(c_1), \cdots ,i^*(c_k) \in H^*(F) \text{ are generators of } H^*(F).$$
Then the following map is an isomorphism:
\begin{align*}
    \Phi: H^*(B) \otimes H^*(F) &\to H^*(E)\\
    \sum_{ij} b_i \otimes i^*(c_j) &\mapsto \sum_{ij} p^*(b_i) \smile c_j
\end{align*}
Furthermore let $b_1, \cdots, b_l$ be generators of $H^*(B)$. Then we have that: $$p^*(b_1), \cdots, p^*(b_l), c_1 \cdots, c_k \text{ are generators of }H^*(E).$$
\end{theorem}
\subsection*{Apply to $M/G \to M/T$}
Let $G$ be a \textbf{semi-simple} Lie group acting on $M$ \textbf{properly} and \textbf{freely} (so that $M/G$ and $M/T$ are smooth). We have the following fibre bundle:
$$G/T \xhookrightarrow{i} M/T \xrightarrow{p} M/G$$
Assume we already have found the generators $b_1, \cdots, b_l$ for $M/G$. We want to apply the Leray-Hirsch theorem to get generators of $M/T$, so we need to find those $c_1, \cdots, c_k \in H^*(E)$.\\
First, let us review some results about $H^*(G/T)$.  We are only interested in $\mathbb{R}$ or $\mathbb{Q}$ coefficients (see Theorem 8.3 on \cite{spec}).
We can consider the following fibre bundle: $$G/T \xhookrightarrow{i} BT \xrightarrow{Bp} BG$$
\begin{theorem}
\label{G/T}
Let $T$ be the maximal torus of $G$. Let $k = \mathbb{Q}$ or $k = \mathbb{R}$. We have the following exact sequence:
\[ k \to H^*(BG;k) \xrightarrow{Bp^*} H^*(BT;k)\xrightarrow{i^*} H^*(G/T;k) \to k\]
Explicitly, $i^*$ is onto and $ker(i^*)$ is generated as an ideal by $Bp^*(BG;k).$
\end{theorem}
More explicitly, assume $dim(T) = rank(G) = l$. We have the following calculation:
\begin{align*}
    H^*(BT) &= H^*(B(U(1)^l)) = H^*(BU(1)^l)\\
    & = \bigotimes_{i=1}^l H^*(BU(1)) = \bigotimes_{i=1}^l H^*(\mathbb{C}P^{\infty})\\
    & = \mathbb{R}[\alpha_1, \cdots, \alpha_l]
\end{align*}
Applying some results in Atiyah and Bott \cite{equivariant} section 2, we have:
\begin{align*}
    H^*(BG) &= H^*_G(pt) = H^*_T(pt)^W\\
    & = \mathbb{R}[\alpha_1, \cdots, \alpha_l]^W\\
    H^*(G/T) &= H^*(BT)/H^*(BG)\\
    &=  \mathbb{R}[\alpha_1, \cdots, \alpha_l]/ \mathbb{R}[\alpha_1, \cdots, \alpha_l]^W
\end{align*}
\begin{corollary} \label{cor4.1}
    $H^*(G/T)$ is generated by $i^*(\alpha_1),\cdots, i^*(\alpha_l)$ and the relation is provided by $H^*(BT)$ through the action of the  Weyl group.
\end{corollary}

\begin{remark} \label{rk4}
$H^*(G/T)$ is generated only by degree $2$ classes.
\end{remark}
The above maps can be considered as pullbacks using the universal property (see section 16.5 in \cite{classify} or \cite{milnor}):\\
$$\begin{tikzcd}
G \arrow[r, "\Tilde{i}"] \arrow[d]
& EG \arrow[d, "\pi"] & H^*(G) & H^*(EG) \arrow[l, "\Tilde{i}^*"]\\
G/T \arrow[r, "i"] \arrow[d]
& BT \arrow[d, "Bp"] & H^*(G/T) \arrow[u] & H^*(BT) \arrow[l , "i^*"] \arrow[u, "\pi^*"]\\
pt \arrow[r, "i_0"] & BG
& H^*(pt) \arrow[u] & H^*(BT) \arrow[l] \arrow[u, "Bp^*"]
\end{tikzcd}$$
Corollary \ref{cor4.1} states that the map $i^*$ maps generators of $H^*(BT)$ to generators of $H^*(G/T)$ surjectively.
Again, using the universal property we can pull back $M \to M/T \to M/G$.  We have:
$$\begin{tikzcd}
G \arrow[r, "\Tilde{i}'"] \arrow[d]
& M \arrow[r, "\Tilde{i}''"] \arrow [d]
& EG \arrow[d, "\pi"]\\
G/T \arrow[r, red , "i'" red] \arrow[d] & M/T \arrow[r, "i''"] \arrow[d, red]
& BT \arrow[d, "Bp"] \\
pt \arrow[r, "i_0'"] & M/G \arrow[r, "i_0''"] & BG
\end{tikzcd}$$
Here, $\Tilde{i}'',i'',i_0 ''$ are pullback maps, and $i'$ is the fibre inclusion from $G/T \to M/T$. Observe that $i''\circ i'$ is also the same fibre inclusion as $i$ in previous diagram.\\
\begin{remark}
Two of the  arrows in the diagram   correspond to the fibre bundle $G/T \to M/T \to M/G$ that we are interested in. Our goal is to find cohomology classes in $H^*(M/T)$, so that the pullback by $i'$ are generators of $H^*(G/T)$.
\end{remark}
From the diagram and combining with the Leray-Hirsch theorem we can easily obtain the following results:
\begin{corollary} \label{cor4.1.2}
The fibre bundle $G/T \xhookrightarrow{i'} M/T \xrightarrow{p} M/G$ satisfies the condition of the Leray-Hirsch theorem.\\
Let $\alpha_1, \cdots, \alpha_l$ be the generators of $H^*(BT)$. We have that $i''^*(\alpha_1), \cdots, i''^*(\alpha_l) \in H^*(M/T)$ are the corresponding $c_i$ in the Leray-Hirsch theorem. 
\end{corollary}
\begin{proof}
$i'^*(i''^*(\alpha_i)) = (i'' \circ i')^*(\alpha_i) = i^*(\alpha_i)$, and by Corollary \ref{cor4.1} they are all generators of $H^*(G/T)$
\end{proof}
\section{Using Chern-Weil homomorphism to express the generators} \label{s5}
Our goal now is to express $i''^*(\alpha_i) \in H^*(M/T)$ and $i^*(\alpha_i) \in H^*(G/T)$ in terms of differential forms using the Chern-Weil homomorphism.   It is natural in the following sense:
\begin{remark} \label{rk5.1}
The pullback map $i''^*: H^*(BT) \to H^*(M/T) = H^*_T(M)$ since $G$ acts  freely and properly on $M$. In the Cartan model its elements are in the class of $\Omega^*(M) \otimes S(\mathfrak{t}^*)$. On the other hand, by definition $H^*(BT) = H^*_T(pt) = S(\mathfrak{t}^*)$. Notice that the invariant polynomials of $\mathfrak{t}$ are $$Inv(\mathfrak{t}) = S(\mathfrak{t}^*)).$$
So $i''^*$ is the following natural map (which is the pullback of the constant map $M \to p$):
\begin{align*}
    i''^*: S(\mathfrak{t}^*) = H^*(BT) &\to H^*(M/T) = H^*_T(M)\\
    f & \mapsto [1 \otimes f] \in H^*_T(M)
\end{align*}
\end{remark}

Now, we need to write out the generators of the symmetric polynomials $S(\mathfrak{t}^*)$. Since our $G$ is assumed to be semi-simple, we have a inner product structure on $\mathfrak{t}$ (the Killing form for $U(n)$ and $SU(n)$ case, it is just the trace). We denote it as:
$$\langle \cdot , \cdot \rangle : \mathfrak{g} \times \mathfrak{g} \to \mathbb{R}$$
Fix an orthonormal basis of $\mathfrak{t}$ in the form  $\{e_1, \cdots, e_l\}$.  We can write out the generators of 
$S(\mathfrak{t}^*)$ by $p_1, \cdots ,p_l \in \mathfrak{t}^* \subset S(\mathfrak{t}^*)$ as:
$$p_i := \langle e_i, \cdot \rangle$$
Now we have $H^*(BT) = S(\mathfrak{t}^*) = \mathbb{R}[f_1, \cdots, f_l]$. We can consider the Chern-Weil homomorphism of $p_i$.\\
First let us compute the generators of $H^*(G/T)$ explicitly by picking a connection using the Maurer-Cartan form $\theta \in \Omega^*(G, \mathfrak{g})$. We require that our connection for $G \to G/T$ is $\mathfrak{t}$ valued. We can use the inner product structure to perform a projection of $\mathfrak{g}$ to $\mathfrak{t}$. We have the following lemma:
\begin{lemma} The following 1-form is a connection 1-form for $G \xrightarrow{q} G/T$:
\begin{equation}
    \theta_T := \sum_{i = 1}^l \langle \theta, e_i\rangle e_i = \sum_{i = 1}^l \theta_i e_i \in \Omega^1(G, \mathfrak{t})
\end{equation}
(where we define $\theta_i = \langle \theta, e_i\rangle$)
\end{lemma}
\begin{proof}
We need to show two things: i) $\theta_T$ is $T$-equivariant, ii) $\theta_T(\Tilde{X}) = X$ for $X \in \mathfrak{t}$ and $\Tilde{X}$ is the fundamental vector field generated by $X$.\\
To prove i) take an arbitrary $g \in T$. We have:
\begin{align*}
    R_g^*(\theta_T) &= \sum_{i = 1}^l \langle R_g^*\theta, e_i\rangle e_i\\
    & = \sum_{i = 1}^l \langle Ad_{g^{-1}} \circ \theta, e_i\rangle e_i = \sum_{i = 1}^l \langle \theta, Ad_g e_i\rangle e_i\\
    & = \sum_{i = 1}^l \langle \theta, e_i\rangle e_i \qquad \text{ (since $g \in T, e_i\in \mathfrak{t}$)}\\
    & = \theta_T = Ad_{g^{-1}} \theta_T \text{ (since $g \in T$ and $\theta_T$ is $\mathfrak{t}$ valued)}
\end{align*}
To prove ii) take arbitrary $X \in \mathfrak{t}$. We have:
\begin{align*}
    \theta_T(\Tilde{X}) &=\sum_{i = 1}^l \langle \theta(\Tilde{X}), e_i\rangle e_i\\
    & = \sum_{i = 1}^l \langle X, e_i\rangle e_i\\
    & = X \qquad \text{ (since $X\in \mathfrak{t}$)}
\end{align*}
So we have shown $\theta_T$ is a connection 1-form.
\end{proof}
Further we can compute the curvature of $\theta_T$:
\begin{align*}
    \Omega_T & = d \theta_T + \frac{1}{2}[\theta_T, \theta_T]\\
    & = d  \sum_{i = 1}^l \langle \theta, e_i\rangle e_i + 0 \text{ (since $\theta_T$ is $\mathfrak{t}$ valued)}\\
    & = \sum_{i = 1}^l \langle d\theta, e_i\rangle e_i = \sum_{i = 1}^l d\theta_i e_i
\end{align*}
Now, we can apply the Chern-Weil homomorphism to it:
\begin{align}
    p_j(\Omega_T) & = p_j(\sum_{i = 1}^l d\theta_i e_i)\nonumber \\
    & = \langle \sum_{i = 1}^l d\theta_i e_i, e_j \rangle \nonumber\\
    & = d\theta_j
\end{align}
By the Chern-Weil homomorphism we have (let $q: G \to G/T$ be the quotient map):
$$\exists \Lambda_i \in H^*(G/T), q^*(\Lambda_i) = d \theta_i$$
\begin{lemma} \label{l5.2}
    $[\lambda_i] \in H^2(G/T)$ are non-zero (in another words the $\Lambda_i$ are not exact) when $G$ is simply connected (e.g. $G = U(n)$).
\end{lemma}
\begin{proof}
    We prove this by contradiction. Assume there exist $\lambda_i \in H^1(G/T)$ such that $\Lambda_i = d \lambda_i$.\\
    Using the quotient map $q$ we can pull back both $\lambda_i$ and $\Lambda_i$ to $H^*(G)$. We have:
\begin{align*}
    d(q^*(\lambda_i) - \theta_i) &= q^*(d \lambda_i) - d\theta_i\\
    &= q^*(d \lambda_i) - q^*(\Lambda_i)\\
    &= q^*(d \lambda_i - \Lambda_i) = 0
\end{align*}
Now let the inclusion $S^1 \xhookrightarrow{} T \xhookrightarrow[]{} G$ be the inclusion of the $i$-th copy of the circle in $T$. For any $v \in TS^1 \subset TG$, with the inclusion we have $\theta(v) = ||v|| e_i$ so $\theta_i(v) = ||v||$.\\
Since $q^*(\lambda_i)$ is a  pullback form, it is horizontal (i.e. $q^*(\lambda_i)(v) = 0$ for $v \in TS^1$).
We have:
\begin{align*}
    \int_{S^1}  q^*(\lambda_i) - \theta_i &= -\int_{S^1} \theta_i = -1
\end{align*}
However, we can consider an arbitrary 2-submanifold $N$ in $G$, such that $\partial N = S^1$. We have:
\begin{align*}
    \int_{S^1}  (q^*(\lambda_i) - \theta_i) &= \int_{\partial N}  (q^*(\lambda_i) - \theta_i)\\
    & = \int_{N} d( q^*(\lambda_i) - \theta_i) = 0
\end{align*}
We have obtained a contradiction. Thus no $\Lambda_i$ is exact.
\end{proof}
\begin{prop}
\label{generatorcw} The classes
$[\Lambda_i]$ constructed above using the Chern-Weil homomorphism are generators of $H^*(G/T)$.
\end{prop}
\begin{proof}
    First we observe all $\Lambda_i$ are 2-forms.  Using the corollary \ref{cor4.1} and Remark \ref{rk4}, we only need to show that the $[\Lambda_i]$ correspond to non-zero cohomology classes and the Weyl group action on them is the same as its action on $\alpha_1, \cdots, \alpha_l \in H^*(BT)$.\\
    From Lemma \ref{l5.2} we have shown the $[\Lambda_i]$ are non-zero, and the action of $G$ on the classes $\Lambda_i$ is exactly the action on $p_i \in S(\mathfrak{t}^*) = H^*(BT)$.
\end{proof}
Now, we can introduce our main result:
\begin{prop}
\label{final}
Pick a connection $\theta_M \in \Omega^1(M,\mathfrak{t})$ and its curvature $\Omega_M \in \Omega^2(M, \mathfrak{t})$ for the principal bundle $M \to M/T$. The Chern-Weil classes $[\Tilde{\Lambda}_i]$ correspond to invariant polynomials $p_i$. More explicitly, $p_i(\Omega_M)$ are the required classes $c_i$ in the Leray-Hirsch theorem.
\end{prop}
\begin{proof}
    All we need to prove is for the map $i$ in the following commutative diagram (it is an arbitrary fibre map), that $i^*([\Tilde{\Lambda_i]}) = [\Lambda_i]$.
    $$\begin{tikzcd}
G \arrow[r, "\Tilde{i}"] \arrow[d]
& M \arrow[d, "q"]\\
G/T \arrow[r, "i"] \arrow[d]
& M/T \arrow[d, "p"]\\
pt \arrow[r, "i_0"] & M/G
\end{tikzcd}$$
The main observation to make is that $\Tilde{i}^*(\theta_M) \in \Omega^1(G, \mathfrak{t})$ is also a connection 1-form.
\begin{align*}
    \Tilde{i}^*(p_j(\Omega_M)) &= p_j(\Tilde{i}^*(\Omega_M))
\end{align*}
Thus $i^*([\Tilde{\Lambda}_j])$ are the Chern-Weil classes of $p_j$ using the connection $\Tilde{i}^*(\theta_M)$, and the Chern-Weil homomorphism is independent of the connection. Finally we obtain:
$$i^*([\Tilde{\Lambda}_j]) = \Lambda_j$$
We have proven that $[\Tilde{\Lambda}_j]$ satisfies the conditions in the Leray-Hirsch theorem.
\end{proof}

\section{Application to moduli space of parabolic bundles}
Finally, we apply our result to the moduli space of parabolic bundles. We need to fit the moduli space of parabolic bundles and moduli space of flat connections into the above setting using equation (\ref{eq1}).\\
We denote $Y_{g} := \{A_1, \cdots, B_g | [A_1, B_1]\cdots [A_g,B_g] = g\}$. Thus we have that the moduli space of parabolic bundles is $Y_{\beta}/T$, and we know the generators in $Y_{\beta}/G$. We can apply Proposition \ref{final} to $G/T \xhookrightarrow{i} Y_\beta/T \xrightarrow{p} Y_\beta /G$.\\
\begin{corollary} \label{cor6.0.1}
    Let $\theta_Y \in \Omega^*(Y_\beta, \mathfrak{t})$ be the connection of $Y_\beta \to Y_\beta /T$.
    Further assume $[\Lambda_i] \in H^2(Y_\beta/T)$ are the Chern-Weil classes corresponding to $p_i$ where $p_i$ generates $Inv(\mathfrak{t})$. Then the generators of $H^*(Y_\beta/T)$ are:
    $$p^*(a_r),p^*(b_r^j),p^*(f_r),\Lambda_i$$
where $a_r,b_r^j,f_r$ are generators of $H^*(Y_\beta/G)$ in \cite{groupcoho}.
\end{corollary}

\subsection*{Explicit connection for parabolic bundles}
Finally, in order to find explicit generators in terms of differential forms, we will write out the connection of $Y_{\beta exp(\Lambda)} \to Y_{\beta exp(\Lambda)}/T$ using Maurer-Cartan forms.\\
For regular $\Lambda \in \mathfrak{g}$ (for simplicity one may assume that $\Lambda$ is diagonal with distinct nonzero eigenvalues), let $T = Stab_\Lambda$ be a maximal torus of $G$. We have the following identification from section \ref{s2}:
$$\mathcal{M}_{\beta,1}(\Lambda) = \{A_1, \cdots, B_g | [A_1, B_1]\cdots [A_g,B_g] \in \beta exp(\Lambda)\}/T$$
We want to calculate the vertical tangent vectors of the principal bundle (i.e. tangent vectors of fibres) at point $(A_1, \cdots, B_g)$. Since it is a principal bundle, all we need are the tangent vectors generated by the Lie algebra using  the Lie group action. By explicit calculation we obtain:
$$\frac{d}{dt}|_{t=0}exp(tX)A_i exp(-tX) = XA_i - A_iX.$$
Similarly for $B_i$, combining all terms we get:
\begin{prop}
    The vertical direction of the principal bundle $Y_{\beta exp(\lambda)} \to Y_{\beta exp(\lambda)}/T$ are of the form (for $X \in T$):
    $$(XA_1-A_1X, XB_1-B_1X, \cdots, X B_g -B_g X)$$
\end{prop}
Using the left multiplication map, we can identify the above vector as an element of $\mathfrak{g}^{2g}$:
$(A_1^{-1} XA_1- X, \cdots, B_g^{-1}X B_g - X)$.
Picking an orthonormal basis $e_1, \cdots,e_{n} \in Lie(T)$ for $Lie(T)$,  we will have the following set of vectors:
$$(A_1^{-1} e_i A_1- e_i, \cdots, B_g^{-1} e_i B_g - e_i).$$
In order to see that this set of vectors is linearly independent we need to study the map $X \to Ad_A X - X$ using a root space decomposition.
\begin{prop}
    Let $A \in G$ and $A \not\in T$. Consider the function $f(X) = Ad_A X - X$. Then $\{f(e_i)|i=1,\cdots, n\}$ is a linearly independent set of vectors (where $e_i$ are basis for $\mathfrak{t} = Lie(T)$).\\
\end{prop}

\begin{proof}
    Since $A \not\in T$, then we can assume $A \in T_A$ for some maximal torus. Then we will consider the root decomposition for $\mathfrak{t}_A$:
    $$\mathfrak{g} = \bigoplus_{\alpha \in \mathfrak{t}_A^*} \mathfrak{g}_\alpha,\text{ where } \mathfrak{g}_\alpha = \{X \in \mathfrak{g}| [Y,X] = \alpha(Y)X \text{ for $Y \in \mathfrak{t}_A^*$} \}$$
    Without loss of generality, let $A = exp(Y)$ for some $Y \in \mathfrak{t}_A^*$; by assumption we have $Y \not\in \mathfrak{t} = Lie(T)$.\\
    For $X \in \mathfrak{g}_\alpha$, let $V(t) := Ad_{exp(tY)} X$ be a curve on $\mathfrak{g}$ such that $X(0) = X, X(1) = Ad_A X$. By differentiating with respect to $t$ we obtain:
    $$\frac{d}{dt} V(t) =\frac{d}{dt} Ad_{exp(tY)} X = [Y ,Ad_{exp(tY)} X] = [Y , V(t)]$$
    For $t = 0$, we have $\frac{d}{dt} V(0) = [Y , X] = \alpha(Y)X$ and $[Y, cX] = c\alpha(Y)X$.  For a vector $cX$ in the  $X$ direction,  the result $[Y, cX]$  is still along the $X$ direction. So we obtain $V(t) = c(t)X$ since the direction of $V$ will not change. Thus we can simplify the above differential:
    \begin{align*}
      \frac{d}{dt} c(t)X &= [Y , c(t)X] = c(t)\alpha(Y)X\\
      \frac{d}{dt} c(t) &= |[Y , c(t)X]|/|X| = c(t)\alpha(Y)\\
      c(0) &= 1
    \end{align*}
    By solving the above equation we have $c(t) = exp(t \alpha(Y))$, $Ad_A X = V(1) = exp(\alpha(Y))X$ then $f(X) = (exp(\alpha(Y))-1)X$.
    Notice that for $X \in \mathfrak{g}_\alpha$, $\alpha(Y) \neq 0$ when $[Y,X] \neq 0$. By the above calculation, the restriction map $f: \mathfrak{g}_\alpha \to \mathfrak{g}_\alpha$ is just scaling by $(exp(\alpha(Y))-1)$, and
    this scaling factor  is nonzero. So the following restriction map is a bijection, mapping linearly independent vectors to linearly independent vectors:
    $$f: \bigoplus_{\alpha \neq 0} \mathfrak{g}_\alpha \to \bigoplus_{\alpha \neq 0} \mathfrak{g}_\alpha$$
    We assume that $Y \not\in \mathfrak{t}$, in other words, $\mathfrak{t} \subset \bigoplus_{\alpha \neq 0} \mathfrak{g}_\alpha$ (Notice the root spaces $\mathfrak{g}_\alpha$ are root spaces associated to a maximal torus containing $Y$, not $\mathfrak{t}$). Since $\{e_i | i =1, \cdots, n\}$ are linearly independent in $\mathfrak{t}$, $f(e_i)$ are also linearly independent.
\end{proof}
We denote $f_\rho : \mathfrak{g} \to \mathfrak{g}^{2g}$ by
$f_\rho(X) = (A_1^{-1} X A_1- e_i, \cdots, B_g^{-1} X B_g - e_i)$. Since we have $[A_1,B_1]\cdots [A_g, B_g] = \beta \neq e$, there exist $A_i$ or $B_i$ that is not in $T$. So from proposition \ref{final} we have that the restriction map $\bar{f}_\rho : \mathfrak{t} \to f_\rho(\mathfrak{t})$ is a bijection.\\
The linearly independent vectors $f_\rho(e_i)$ describing the vertical direction of $Y_\beta \to Y_\beta/T$, we can perform the Gram–Schmidt process to obtain an orthonormal basis $\{w_i | i = 1 ,\cdots n\}$ for the vertical direction. Since $w_i$ are linear combinations of $f_\rho(e_1), \cdots, f_\rho(e_n)$, $w_i \in f_\rho(\mathfrak{t})$,   we have the following natural connection (as a projection from  vectors to the vertical direction):
\begin{align}
    \omega &= \sum_{i = 1,\cdots, g; j = 1,\cdots, n} \langle \theta_{A_i}, w_j \rangle \bar{f}_\rho^{-1} (w_i) + \langle \theta_{B_i}, B_i^{-1} w_j \rangle \bar{f}_\rho^{-1}(w_i)
\end{align}
    $\theta_{A_i} = pr^*_{A_i} (\theta)$ (where $pr_{A_i}: Y_{\beta exp(\Lambda)} \to G$ is the projection map to the space spanned by $A_i$). Similarly, $\theta_{B_i} = pr^*_{B_i} (\theta) $ (where $pr_{B_i}: Y_{\beta exp(\Lambda)} \to G$ is the projection map to the space spanned by $B_i$).
So the additional generators for cohomology of parabolic bundles are $\Lambda_i = \langle d\omega + \frac{1}{2} [\omega \wedge \omega], e_i\rangle$ (which is completely expressed in terms of Maurer-Cartan forms similar to the other generators in \cite{groupcoho}).

\section{Intersection Pairing}
\subsection{Fibre integration and push-forward map}
First, we fix some notation for fibre integration. Let $\pi : T \to M$ be a fibre bundle with fibre $F$. For simplicity we assume $M, F, T$ are all compact. Let $\alpha \in \Omega^*(M)$ and $\beta \in \Omega^*(T)$ such that for all fibre inclusion maps $i:F \to T$, $i^*(\beta)$ represents the same cohomology class. Then we have:
$$\int_T \beta \wedge \pi^*(\alpha) = \int_M \left( \int_F i^*(\beta) \right) \alpha = \left( \int_M \alpha \right) \left( \int_F i^*(\beta) \right)$$
We observe that our generators for the moduli space of parabolic bundles are either pullbacks of generators on the base space or forms satisfying the Leray-Hirsch condition (recall Theorem \ref{LH}) corresponding to $\beta$. As a result, the intersection pairing of the moduli space of parabolic bundles can be decomposed into intersection pairings of the moduli space of parabolic bundles and intersection pairing of the fibres.\\

\subsection{Intersection pairing of $G/T$}
Let $X \in \mathfrak{t}$ be a regular element in Lie algebra of maximal torus $T$, such that the orbit $\mathcal O_X = \{g^{-1}Xg | g \in G\} \cong G/T$ and we have the following explicit map:
\begin{align*}
    h: G/T &\to \mathcal O_X\\
    Tg &\mapsto g^{-1} X g
\end{align*}
Notice that we let the $G/T = \{Tg|g \in G\}$ be the right cosets. Let $\theta$ denote the right Maurer-Cartan form for $G$. We also identify the orbit of $\mathcal O_X$ with the coadjoint orbit of $X^* = \langle X, -\rangle \in \mathfrak{t}$ using the Killing form.
Let's compare the KKS form $\omega_{KKS}$ on $\mathcal O_X$ with $d\langle X, \theta \rangle$ on $G/T$.
\begin{prop}

    For $\omega_{KKS}$ on the orbit space $\mathcal O_X$, we have $h^*(\omega_{KKS}) =  -d\langle X, \theta \rangle$.
\end{prop}
\begin{proof}
    We first consider the right equivariant $G$ action on $\mathcal O_X$ and $G/T$:
    $$G/T \curvearrowleft G: (T g).g_1 =T (g g_1) $$
    $$\mathcal O_X \curvearrowleft G: (g^{-1} X g).g_1 = (g g_1)^{-1} X (g g_1)$$
    Notice that the right Maurer-Cartan form $\theta$ is invariant under the right Lie group action. By definition of the KKS form, $\omega_{KKS}|_{g^{-1} X g} (Y_1^{\#}, Y_2^{\#}) = \langle g^{-1} X g, [Y_1, Y_2]\rangle$, where $Y_1,Y_2 \in \mathfrak{g}$ and $Y_1^{\#}, Y_2^{\#}$ are fundamental vector fields generated by $Y_1, Y_2$ using the right $G$ action. We have:
    \begin{align*}
        &g^* (\omega_{KKS}|_{g^{-1}Xg}) (Y_1^{\#}|_X, Y_2^{\#}|_X)\\
        &= \omega_{KKS}|_{g^{-1}Xg} (g_* (Y_1^{\#}|_X), g_*(Y_2^{\#}|_X))\\
        &= \omega_{KKS}|_{g^{-1}Xg} (\left.\frac{d}{dt}\right|_{t = 0} g^{-1} exp(-t Y_1) X exp(t Y_1) g,\\
        &\qquad \left.\frac{d}{dt}\right|_{t = 0} g^{-1}exp(-t Y_2)X exp(t Y_2) g)\\
        &= \omega_{KKS}|_{g^{-1}Xg} (g^{-1} (Y_1^{\#}|_X)g, g^{-1}(Y_2^{\#}|_X)g)\\
        &= \langle g^{-1}Xg, [g^{-1} Y_1 g, g^{-1}Y_2g]\rangle\\
        &= \langle X, [Y_1, Y_2]\rangle = \omega_{KKS}|_X (Y_1^{\#}|_X, Y_2^{\#}|_X)
    \end{align*}
    We obtain that $g^*(\omega_{KKS}) = \omega_{KKS}$, in other words it is invariant under the right group action. In order to show $h^*(\omega_{KKS}) = \langle X, \theta \rangle$, we only need to show that $h^*(\omega_{KKS}|_{X}) = d\langle X, \theta \rangle|_{Te}$. Consider $Y_1, Y_2 \in \mathfrak{g}$, and let $Y_i^{\#}, \widetilde{Y_i}$ be the fundamental vector fields on $\mathcal O_X$ and $G/T$. On the one hand, we have $\mathcal{L}_{\widetilde{Y_1}} \langle X, \theta(\widetilde{Y_2}) \rangle= \mathcal{L}_{\widetilde{Y_1}} \langle X, Y_2 \rangle = 0$. On the other hand we can see using Cartan's formula:
    \begin{align*}
        \left.\mathcal{L}_{\widetilde{Y_1}}\right|_{Te} \langle X, \theta(\widetilde{Y_2}) \rangle &= (d \langle X, \theta(-) \rangle)(\widetilde{Y_1},\widetilde{Y_2}) + \langle X, \theta(\mathcal{L}_{\widetilde{Y_1}} (\widetilde{Y_2})) \rangle\\
        &= \langle X, d\theta(-) \rangle)(\widetilde{Y_1},\widetilde{Y_2}) + \langle X, \theta( [ \widetilde{Y_1}, \widetilde{Y_2}]) \rangle\\
        &= \langle X, d\theta(\widetilde{Y_1},\widetilde{Y_2}) \rangle) + \langle X, [Y_1,Y_2] \rangle\\
        &= (d\langle X, \theta\rangle + h^*(\omega_{KKS})) (\widetilde{Y_1}, \widetilde{Y_2})
    \end{align*}
    We finally obtain that $d\langle X, \theta\rangle + h^*(\omega_{KKS}) = 0$.
\end{proof}
Let us take a closer look at the $U(n)$ case now (since $U(n)/T_{U(n)} = SU(n)/T_{SU(n)}$, the result also applies to $SU(n)$ case). We find the following explicit orthonormal basis for $\mathfrak{u}(n)$:
\begin{align*}
    e_j &= [(E_j)_{kl}], (E_j)_{kl} =\begin{cases}
        \frac{i}{\sqrt{2}},& k=l=j\\
        0, &\text{ otherwise}
    \end{cases} =
\bbordermatrix{ & & \text{\scriptsize j-th} &\cr
       & & &\cr
      \text{\scriptsize j-th}& & \frac{i}{\sqrt{2}} & \cr
       &  & &}\\
    u_{jk} &= [(U_{jk})_{lp}], (U_{jk})_{lp} =\begin{cases}
        \frac{i}{\sqrt{2}},& j=l,k=p\\
        \frac{i}{\sqrt{2}},& j=p,k=l\\
        0, &\text{ otherwise}
    \end{cases}=
    \bbordermatrix{ & & \text{\scriptsize k-th}&  \text{\scriptsize j-th} &\cr
       & & & &\cr
      \text{\scriptsize k-th}& &  &\frac{i}{\sqrt{2}} & \cr
       \text{\scriptsize j-th} & & \frac{i}{\sqrt{2}} & & \cr
       &  & & &}\\
    v_{jk} &= [(V_{jk})_{lp}], (U_{jk})_{lp} =\begin{cases}
        \frac{1}{\sqrt{2}},& j=l,k=p\\
        \frac{-1}{\sqrt{2}},& j=p,k=l\\
        0, &\text{ otherwise}
    \end{cases}=
    \bbordermatrix{ & & \text{{\scriptsize k-th}}&  \text{\scriptsize j-th} &\cr
       & & & &\cr
      \text{\scriptsize k-th}& &  &\frac{1}{\sqrt{2}} & \cr
       \text{\scriptsize j-th} & & \frac{-1}{\sqrt{2}}& & \cr
       &  & & &}\\
    &\text{(where $j>k$)}
\end{align*}
Notice that  the $e_j$ provide a basis for $\mathfrak{t}$. Consider the tangent space of the homogeneous space $T_{Te} G/T = \mathfrak{g}/\mathfrak{t} = \mathfrak{t}^{\perp}$ which is spanned by $\{u_{jk},v_{jk}\}$. Consider the local coordinates $(x_{jk},y_{jk})$ for $G/T$ near $Te$, such that $\left. \frac{\partial}{\partial x_{jk}}\right|_{Te} = u_{jk}, \left. \frac{\partial}{\partial y_{jk}}\right|_{Te} = v_{jk}$.\\
It is easy to do the following calculation using  Proposition \ref{final}:
\begin{align*}
    d\langle e_i, \theta \rangle (\partial/\partial x_{jk},\partial/\partial x_{lm}) &= \langle e_i, [u_{jk}, u_{lm}] \rangle = 0\\
    d\langle e_i, \theta \rangle (\partial/\partial y_{jk},\partial/\partial y_{lm}) &= \langle e_i, [v_{jk}, v_{lm}] \rangle = 0\\
    d\langle e_i, \theta \rangle (\partial/\partial x_{jk},\partial/\partial y_{lm}) &= \langle e_i, [v_{jk}, u_{lm}] \rangle = \begin{cases}
        1, i = k = m, j = l\\
        -1, i = j=l, k = m\\
        0, \text{otherwise}
    \end{cases}
\end{align*}
We obtain locally near $Te$:
$$d \theta_i = d\langle e_i, \theta \rangle = \sum_{k > i}dx_{ki}\wedge dy_{ki} - \sum_{j < i} dx_{ji}\wedge dy_{ji}$$
Let $X = \sum a_i e_i$. We have that:
$$d\langle X, \theta \rangle = \sum_{i=1,\cdots,n} a_i d\langle e_i, \theta \rangle= \sum_{j > k} (a_k - a _j) dx_{jk} \wedge dy_{jk}$$
We can now perform the following calculation for the volume form:
\begin{align}
    (d\langle X, \theta \rangle)^n &= \left ( \sum_{j > k} (a_k - a _j) dx_{jk} \wedge dy_{jk}\right)^n\nonumber\\
    &= \left(\prod_{j>k}(a_k-a_j)\right) \bigwedge_{j>k} dx_{jk} \wedge dy_{jk}\nonumber\\
    & = C_0\prod_{j>k}(a_k-a_j) vol_{G/T}\nonumber\\
    \int_{G/T} (d\langle X, \theta \rangle)^n & = C_1 \prod_{j>k}(a_k-a_j) \label{intpair}
\end{align}
On the other hand, from Proposition 6.1, we know that $(d\langle X, \theta \rangle)^n = \omega_{KKS}^n = vol_{\mathcal O_X}$. By fixing a $X_0 = \sum a'_i e_i$, we can determine the constant $C_1$ above:
\begin{align*}
    \int \left(  \sum_{i=1,\cdots,n} a'_i d\langle e_i, \theta \rangle\right )^n &= \int vol_{\mathcal O_{X_0}} = Vol(\mathcal O_{X_0})\\
    C_1 &= Vol(\mathcal O_{X_0})/ \prod_{j>k}(a'_k-a'_j)
\end{align*}
Considering $a_i$ as a variable, the above intersection pairing formula (\ref{intpair}) is a polynomial with variable $a_i$. In order to retrieve the coefficient for each terms, we can take the differential to eliminate other terms. Using the trick of taking a derivative with respect to  $a_i$, we can obtain the following result:
\begin{align*}
    \frac{\partial}{\partial a_j} \left(\sum_{i=1,\cdots,n}a_i d\theta_i \right)^n &= n \left(\sum_{i=1,\cdots,n}a_i d\theta_i \right)^{n-1}d\theta_j\\
    \frac{\partial^{\alpha_1}}{\partial a_1^{\alpha_1}}\cdots\frac{\partial^{\alpha_n}}{\partial a_n^{\alpha_n}} \left(\sum_{i=1,\cdots,n}a_i d\theta_i \right)^n & = n ! d\theta_1^{\alpha_1}\cdots d\theta_n^{\alpha_n}\\
    \int_{G/T} d\theta_1^{\alpha_1}\cdots d\theta_n^{\alpha_n} &= \frac{1}{n!}\int_{G/T}\frac{\partial^{\alpha_1}}{\partial a_1^{\alpha_1}}\cdots\frac{\partial^{\alpha_n}}{\partial a_n^{\alpha_n}} \left(\sum_{i=1,\cdots,n}a_i d\theta_i \right)^n\\
    & = \frac{C_1}{n!} \frac{\partial^{\alpha_1}}{\partial a_1^{\alpha_1}}\cdots\frac{\partial^{\alpha_n}}{\partial a_n^{\alpha_n}}\prod_{j>k}(a_k-a_j).
\end{align*}
Combined with results on intersection pairings of the moduli space of flat connections, we have  the following well defined intersection pairing formula for arbitrary combinations of generators of the moduli space of parabolic bundles:
\begin{theorem}
For generators from Corollary \ref{cor6.0.1}  and arbitrary polynomials $f(p^*(a_r), p^*(b^j_r), p^*(f_r))$, the intersection pairing formula is
    \begin{align}
  &\int_{\mathcal{M}_{\beta,1}(\Lambda)} f(p^*(a_r), p^*(b^j_r), p^*(f_r)) \Lambda_1^{\alpha_1} \cdots \Lambda_n^{\alpha_n}\nonumber\\
    = \quad& \frac{C_1}{n!} \frac{\partial^{\alpha_1}}{\partial a_1^{\alpha_1}}\cdots\frac{\partial^{\alpha_n}}{\partial a_n^{\alpha_n}}\prod_{j>k}(a_k-a_j)\int_{\mathcal{M}_{\beta}} f(a_r, b^j_r, f_r) \label{intpair2} 
\end{align}

\end{theorem}
The intersection pairing over $\mathcal{M}_\beta$ is well studied, so $\int_{\mathcal{M}_{\beta}} f(a_r, b^j_r, f_r)$ is well studied.

\end{document}